\documentclass[10pt]{amsart}
\usepackage{amscd}
\usepackage{amssymb}
\usepackage[all]{xy}
\usepackage{lmodern}
\usepackage[T1]{fontenc}
\date{11 June 2013}
\title{Cohomologically Cofinite Complexes}
\usepackage{hyperref}
\hypersetup{colorlinks=false}

\author{Marco Porta, Liran Shaul and Amnon Yekutieli}
\address{Porta: Department of  Mathematics,
Ben Gurion University, Be'er Sheva 84105, Israel}
\email{marcoporta1@libero.it}

\address{Shaul: Department of Mathematics,
The Weizmann Institute of Science, Rehovot 76100, Israel}
\email{liran.shaul@weizmann.ac.il}

\address{Yekutieli: Department of  Mathematics,
Ben Gurion University, Be'er Sheva 84105, Israel}
\email{amyekut@math.bgu.ac.il}

\thanks{
Primary: 13D07; Secondary: 13B35, 13C12, 13D09, 18E30.}
\thanks{This research was supported by the Israel Science Foundation and the
Center for Advanced Studies at BGU}

\newtheorem{thm}[equation]{Theorem}
\newtheorem{cor}[equation]{Corollary}
\newtheorem{prop}[equation]{Proposition}
\newtheorem{lem}[equation]{Lemma}
\theoremstyle{definition}
\newtheorem{dfn}[equation]{Definition}
\newtheorem{rem}[equation]{Remark}
\newtheorem{exa}[equation]{Example}

\newtheorem{setup}[equation]{Setup}
\numberwithin{equation}{section}

\newcommand{\xar}{\xrightarrow}
\newcommand{\opn}{\operatorname}
\newcommand{\cat}[1]{\operatorname{\mathsf{#1}}}

\newcommand{\rmitem}[1]{\item[\text{\textup{(#1)}}]}
\newcommand{\mfrak}[1]{\mathfrak{#1}}

\newcommand{\mrm}[1]{\mathrm{#1}}
\newcommand{\mbb}[1]{\mathbb{#1}}

\newcommand{\tup}[1]{\textup{#1}}
\newcommand{\bsym}[1]{\boldsymbol{#1}}

\newcommand{\til}[1]{\tilde{#1}}
\newcommand{\what}[1]{\widehat{#1}}

\newcommand{\K}{\mbb{K}}

\newcommand{\N}{\mbb{N}}
\newcommand{\Z}{\mbb{Z}}

\newcommand{\m}{\mfrak{m}}

\newcommand{\Ga}{\Gamma}
\renewcommand{\a}{\mfrak{a}}

\renewcommand{\d}{\mrm{d}}
\newcommand{\ot}{\otimes}

\newcommand{\lb}{\linebreak}

\renewcommand{\ot}{\otimes}

\begin{document}

\begin{abstract}
Let $A$ be a commutative noetherian ring, and $\a$ an ideal in it.
In this paper we continue the study, begun in \cite{PSY1}, of the derived
$\a$-adic completion and the derived $\a$-torsion functors. Here are our
results: (1) a structural characterization of bounded above
{\em cohomologically complete complexes}; (2) the {\em
Cohomologically Complete Nakayama Theorem}; and (3) a characterization of {\em
cohomologically cofinite complexes}.
\end{abstract}

\maketitle

\setcounter{section}{-1}
\section{Introduction}

Let $A$ be a noetherian commutative ring, and let $\a$ be an ideal in it.
(We do not assume that $A$ is $\a$-adically complete.)
There are two operations associated to this data: the {\em $\a$-adic
completion} and the {\em $\a$-torsion}.
For an $A$-module $M$ its $\a$-adic completion is the $A$-module
\[ \Lambda_{\a} (M) = \what{M} := \lim_{\leftarrow i} \,
M / \a^i M . \]
We say that $M$ is an {\em $\a$-adically complete module} if the canonical 
homomorphism 
$M \to \Lambda_{\a} (M)$ is an isomorphism.
An element $m \in M$ is called an $\a$-torsion element if
$\a^i m = 0$ for $i \gg 0$. The $\a$-torsion elements form
the $\a$-torsion submodule $\Gamma_{\a} (M)$ of $M$. We say that $M$ 
is an {\em $\a$-torsion module} if the canonical homomorphism 
$\Gamma_{\a} (M) \to M$ is an isomorphism.

Let us denote by $\cat{Mod} A$ the category of $A$-modules. So we have additive
functors
\[ \Lambda_{\a}, \Gamma_{\a} : \cat{Mod} A \to \cat{Mod} A . \]
The functor $\Gamma_{\a}$ is left exact; whereas $\Lambda_{\a}$ is neither left
exact nor right exact. (Of course the completion functor
$\Lambda_{\a}$ is exact on the subcategory $\cat{Mod}_{\mrm{f}} A$ of finitely
generated modules.) Both functors $\Lambda_{\a}$ and $\Gamma_{\a}$ are 
idempotent, namely for any $M \in \cat{Mod} A$ the module 
$\Lambda_{\a}(M)$ is $\a$-adically complete, and the module 
$\Gamma_{\a}(M)$ is $\a$-torsion. (Note that when $A$ is not noetherian,
the functor $\Lambda_{\a}$ might fail to be idempotent.)

In this paper we continue the study of questions of
homological nature about the functors $\Lambda_{\a}$ and $\Gamma_{\a}$, that was 
started in \cite{PSY1}.

We denote by $\cat{C}(\cat{Mod} A)$ the category of complexes over $A$. 
Recall that a complex $P \in \cat{C}(\cat{Mod} A)$ is called {\em K-projective} 
if for any acyclic complex $N \in \cat{C}(\cat{Mod} A)$, the complex
$\opn{Hom}_{A}(P, N)$ is also acyclic.
A complex $I \in \cat{C}(\cat{Mod} A)$ is called {\em K-injective} if
for any acyclic complex $N$, the complex
$\opn{Hom}_{A}(N, I)$ is also acyclic.
These definitions were introduced in \cite{Sp}. In \cite[Section 3]{Ke} it is
shown that ``K-projective'' is the same as ``having property (P)'', and
``K-injective'' is the same as ``having property (I)''.

The derived category of $\cat{Mod} A$ is denoted by $\cat{D}(\cat{Mod} A)$.
As usual we write $\cat{D}^{\mrm{b}}(\cat{Mod} A)$ for the full subcategory of 
$\cat{D}(\cat{Mod} A)$ consisting of bounded complexes. Less customary is our 
notation $\cat{D}(\cat{Mod} A)^{\mrm{b}}$ for the category of complexes with 
bounded cohomology. Of course the inclusion 
$\cat{D}^{\mrm{b}}(\cat{Mod} A) \subset \cat{D}(\cat{Mod} A)^{\mrm{b}}$ 
is an equivalence; and often the distinction between these categories is 
blurred. However, since much of our work has to do with boundedness conditions, 
we wish to be accurate on this point. Similarly we have equivalences
$\cat{D}^{\star}(\cat{Mod} A)  \subset \cat{D}(\cat{Mod} A)^{\star}$, 
where $\star$ is either $+$ or $-$. 

The derived functors
\[ \mrm{L} \Lambda_{\a}, \mrm{R} \Gamma_{\a} :
\cat{D}(\cat{Mod} A) \to \cat{D}(\cat{Mod} A) \]
exist. The left derived functor $\mrm{L} \Lambda_{\a}$ is constructed using
K-projective resolutions, and the right derived functor $\mrm{R} \Gamma_{\a}$
is constructed using K-injective resolutions.
The functors $\mrm{L} \Lambda_{\a}$ and $\mrm{R} \Gamma_{\a}$ have finite 
cohomological dimensions, and therefore they send 
$\cat{D}(\cat{Mod} A)^{\star}$ into itself, where $\star$ is either
$\mrm{b}$, $+$ or $-$.

A complex $M \in \cat{D}(\cat{Mod} A)$ is called a {\em cohomologically
$\a$-torsion complex} if the canonical morphism
$\mrm{R} \Gamma_{\a} (M) \to M$
is an isomorphism.
The complex $M$  is called a {\em cohomologically $\a$-adically complete
complex} if the canonical morphism
$M \to \mrm{L} \Lambda_{\a} (M)$
is an isomorphism.
We denote by $\cat{D}(\cat{Mod} A)_{\a \tup{-tor}}$ and
$\cat{D}(\cat{Mod} A)_{\a \tup{-com}}$
the full subcategories of $\cat{D}(\cat{Mod} A)$ consisting of
cohomologically $\a$-torsion complexes and cohomologically $\a$-adically
complete complexes, respectively. These are triangulated subcategories.
The {\em MGM Equivalence}, \cite[Theorem 1.1]{PSY1}, says that the
functor
\[ \mrm{R} \Gamma_{\a} : \cat{D}(\cat{Mod} A)_{\a \tup{-com}} \to
\cat{D}(\cat{Mod} A)_{\a \tup{-tor}} \]
is an equivalence of triangulated categories, with quasi-inverse
$\mrm{L} \Lambda_{\a}$. (Note that in the noetherian case, the MGM equivalence  
follows easily from results of \cite{AJL2}.)

In \cite{PSY1} we proved that a complex $M$ is cohomologically torsion, namely
it belongs to $\cat{D}(\cat{Mod} A)_{\a \tup{-tor}}$, if and only if all of its
cohomology modules $\opn{H}^i(M)$ are torsion modules. This is false for
cohomologically complete complexes -- indeed, in Example \ref{exa:1}
we exhibit a cohomologically complete complex $M$ such that $\mrm{H}^i (M) = 0$
for all $i \neq 0$, and the module $\mrm{H}^0 (M)$ is {\em not} $\a$-adically
complete. So the category $\cat{D}(\cat{Mod} A)_{\a \tup{-com}}$ is quite
mysterious.

However we do have a structural characterization of the
subcategory \lb $\cat{D}(\cat{Mod} A)^-_{\a \tup{-com}}$ of
cohomologically complete complexes with bounded above cohomology.
The notion of {\em $\a$-adically projective module} is recalled in
Definition \ref{dfn:4}. (See Theorem \ref{thm:108} for a comparison to the {\em 
formally projective modules} of \cite{EGA-IV}.)  The structure of
$\a$-adically projective modules is well-understood (see Corollary
\ref{cor:51}).
Let us denote by $\cat{AdPr} (A, \a)$ the full subcategory of
$\cat{Mod} A$ consisting of $\a$-adically projective modules. This is an
additive category. There is a corresponding triangulated category
$\cat{K}^-(\cat{AdPr} (A, \a))$, which is a full subcategory of the homotopy 
category $\cat{K}(\cat{Mod} A)$.

Here is the first main result of this paper.

\newpage
\begin{thm} \label{thm:100}
Let $A$ be a noetherian commutative ring, and $\a$ an ideal in $A$. The
localization functor
$\cat{K}(\cat{Mod} A) \to \cat{D}(\cat{Mod} A)$
induces an equivalence of triangulated categories
\[ \cat{K}^-(\cat{AdPr} (A, \a)) \to
\cat{D}(\cat{Mod} A)_{\a \tup{-com}}^- . \]
\end{thm}

This is Theorem \ref{thm:56} in the body of the paper.
In particular we see that a bounded above complex $M$ is cohomologically
complete iff it is isomorphic (in $\cat{D}(\cat{Mod} A)$) to a bounded above
complex of adically projective modules.

Our second main result is influenced by \cite{KS}, as explained in Remark \ref{rem:ks} below.
For the rest of the introduction we assume $A$ is
$\a$-adically complete.

\begin{thm}[Cohomologically Complete Nakayama]  \label{thm:51}
Let $A$ be a noetherian commutative ring, $\a$-adically complete with respect
to an ideal $\a$, and let $A_0 := A / \a$.
Let  $M \in \cat{D}(\cat{Mod} A)_{\a \tup{-com}}$ and $i_0 \in \Z$
be such that $\mrm{H}^i (M) = 0$ for all $i > i_0$, and
$\mrm{H}^{i_0} (A_0 \otimes^{\mrm{L}}_A M)$ is a finitely generated
$A_0$-module. Then $\mrm{H}^{i_0} (M)$ is a finitely generated
$A$-module.
\end{thm}

This is Theorem \ref{thm:23} in the body of the paper. Note that in particular
$\mrm{H}^{0} (M)$ is $\a$-adically complete as $A$-module, in contrast to
Example \ref{exa:1}.

A consequence of Theorem \ref{thm:51} is:

\begin{cor}
Let $M \in \cat{D}(\cat{Mod} A)^-_{\a \tup{-com}}$.
If $A_0 \otimes^{\mrm{L}}_A M = 0$ then $M = 0$.
\end{cor}

See Remark \ref{rem:ks} for an extension of this result to unbounded complexes.

Consider the category $\cat{D}_{\tup{f}}(\cat{Mod} A)^{\mrm{b}}$
of complexes with bounded finitely generated cohomology. It is the closure 
under isomorphisms, inside $\cat{D}(\cat{Mod} A)$, of 
$\cat{D}^{\mrm{b}}_{\tup{f}}(\cat{Mod} A)$. 
It is not hard to see that 
$\cat{D}_{\tup{f}}(\cat{Mod} A)^{\mrm{b}}$ is contained in
$\cat{D}(\cat{Mod} A)_{\a \tup{-com}}$ (see Proposition \ref{prop:263}).
We denote by $\cat{D}(\cat{Mod} A)^{\mrm{b}}_{\a \tup{-cof}}$
the essential image of $\cat{D}_{\tup{f}}(\cat{Mod} A)^{\mrm{b}}$
under the functor $\mrm{R} \Gamma_{\a}$; so
\[ \cat{D}(\cat{Mod} A)^{\mrm{b}}_{\a \tup{-cof}} \subset
\cat{D}(\cat{Mod} A)^{\mrm{b}}_{\a \tup{-tor}}  . \]
The objects of $\cat{D}(\cat{Mod} A)^{\mrm{b}}_{\a \tup{-cof}}$
are called {\em cohomologically $\a$-adically cofinite complexes}.
By MGM Equivalence we have an equivalence of triangulated categories
\[ \mrm{R} \Gamma_{\a} : \cat{D}_{\tup{f}}(\cat{Mod} A)^{\mrm{b}} \to
\cat{D}(\cat{Mod} A)^{\mrm{b}}_{\a \tup{-cof}} , \]
with quasi-inverse $\mrm{L} \Lambda_{\a}$.
This implies that for
$M \in \cat{D}(\cat{Mod} A)^{\mrm{b}}_{\a \tup{-tor}}$
to be cohomologically cofinite it is necessary and sufficient that
$\mrm{L} \Lambda_{\a} (M) \in \cat{D}_{\tup{f}}(\cat{Mod} A)^{\mrm{b}}$.
See Proposition \ref{prop:3}. Yet this last condition is hard to check!

The importance of $\cat{D}(\cat{Mod} A)_{\a \tup{-cof}}^{\mrm{b}}$
comes from the fact that it contains the {\em t-dualizing complexes} (defined in 
\cite{AJL2}). The next theorem (which is Theorem
\ref{thm:cof.1} in the body of the paper, and whose proof uses
Theorem \ref{thm:51}) answers a question we asked in 1998 in
\cite{Ye1}.

\begin{thm} \label{thm:52}
Let $A$ be a noetherian commutative ring, $\a$-adically complete with respect to
an ideal $\a$, and let $A_0 := A / \a$. The following conditions are
equivalent for
$M \in \cat{D}(\cat{Mod} A)^{\mrm{b}}_{\a \tup{-tor}}$~\tup{:}
\begin{enumerate}
\rmitem{i} $M$ is cohomologically $\a$-adically cofinite.
\rmitem{ii} For every $i$ the $A_0$-module
$\opn{Ext}^i_A(A_0, M)$ is finitely generated.
\end{enumerate}
\end{thm}

In Proposition \ref{prop:ha2}, we show that our definition of a cohomologically  
cofinite complex coincides with that of \cite{Ha2} for bounded complexes when 
$A$ is a regular ring.

\medskip \noindent
\textbf{Acknowledgments.}
We wish to thank Joseph Lipman, Ana Jeremias and Leo Alonso
for helpful discussions. We are also grateful to the anonymous referee who read  
the paper carefully and made some useful suggestions.

\section{Structural Results for Cohomologically Complete Complexes}

In this section we assume this setup:

\begin{setup} \label{set:300}
$A$ is a noetherian commutative ring, and $\a$ is an ideal in it.
\end{setup}

We do not assume that $A$ is $\a$-adically complete. We wish to gain a
better understanding of cohomologically $\a$-adically
complete complexes. For this we recall some notation, definitions and results 
from \cite{Ye2}.

Let $Z$ be a set. We denote by $\opn{F}(Z, A)$
the set of all functions $f : Z \to A$. This is an $A$-module.
The subset of finite support functions is denoted by
$\opn{F}_{\mrm{fin}}(Z, A)$; this is a free $A$-module with basis the set
$\{ \delta_z \}_{z \in Z}$ of delta functions.

Let $\what{A} := \Lambda_{\a} (A)$, and let
$\what{\a} := \a \cdot \what{A}$, an ideal of the ring $\what{A}$. Then
$\what{\a} \cong \Lambda_{\a} (\a)$, the ring $\what{A}$ is
$\what{\a}$-adically complete and noetherian, and the homomorphism
$A \to \what{A}$ is flat. Given an element $a \in \what{A}$, its
$\a$-adic order is
\[ \opn{ord}_{\a}(a) := \sup \, \{ i \in \N \mid a \in \what{\a}^{\, i} \}
\in \N \cup \{ \infty \} . \]

\begin{dfn} \label{dfn:decay}
Let $Z$ be a set.
\begin{enumerate}
\item A function $f : Z \to \what{A}$ is called {\em $\a$-adically decaying} if
for every $i \in \N$ the set
\[ \{ z \in Z \mid \opn{ord}_{\a}( f(z) ) \leq i \} \]
is finite.

\item The set of $\a$-adically decaying functions $f : Z \to \what{A}$ is
called the {\em module of decaying functions}, and is denoted by
$\opn{F}_{\tup{dec}}(Z, \what{A})$.

\item An $A$-module is called {\em $\a$-adically free} if it is isomorphic to
$\opn{F}_{\tup{dec}}(Z, \what{A})$ for some set $Z$.
\end{enumerate}
\end{dfn}

Note that $\opn{F}_{\tup{dec}}(Z, \what{A})$ is an $\what{A}$-submodule of
$\opn{F}(Z, \what{A})$. 

\begin{exa}
The ring of restricted formal power series $\what{A}\{t_1,\dots,t_n\}$,  
defined in \cite[Section III.4.2]{CA}, is canonically isomorphic, as an 
$\hat{A}$-module, to $\opn{F}_{\mrm{dec}}(\mbb{N}^n, \what{A})$.
\end{exa}

Recall that an $A$-module $M$ is called $\a$-adically complete if the 
canonical homomorphism 
$\tau_M : M \to \Lambda_{\a}(M)$ is bijective.

\begin{dfn} \label{dfn:4}
An $A$-module $P$ is called {\em $\a$-adically projective} if it has these two
properties:
\begin{enumerate}
\rmitem{i} $P$ is $\a$-adically complete.
\rmitem{ii} Suppose $M$ and $N$ are $\a$-adically complete modules, and
$\phi : M \to N$ is a surjection. Then any homomorphism $\psi : P \to N$
lifts to a homomorphism $\til{\psi} : P \to M$; namely
$\phi \circ \til{\psi} = \psi$.
\end{enumerate}
\end{dfn}

\begin{thm}[{\cite[Section 3]{Ye2}}]  \label{thm:70}
Let $Z$ be a set.
\begin{enumerate}
\item The $A$-module $\opn{F}_{\tup{dec}}(Z, \what{A})$ is the
$\a$-adic completion of $\opn{F}_{\mrm{fin}}(Z, A)$.
More precisely, there is a unique $A$-linear isomorphism
\[ \opn{F}_{\tup{dec}}(Z, \what{A}) \cong
\Lambda_{\a} (\opn{F}_{\mrm{fin}}(Z, A)) \]
that is compatible with the homomorphisms from $\opn{F}_{\mrm{fin}}(Z, A)$.

\item The $A$-module $\opn{F}_{\tup{dec}}(Z, \what{A})$ is flat and
$\a$-adically complete.

\item Let $M$ be any $\a$-adically complete $A$-module, and let
$f : Z \to M$ be any function. Then there is a unique $A$-linear
homomorphism
$\phi : \opn{F}_{\tup{dec}}(Z, \what{A}) \to M$
such that $\phi(\delta_z) = f(z)$ for every $z \in Z$.
\end{enumerate}
\end{thm}

\begin{cor}[{\cite[Corollary 3.15]{Ye2}}]  \label{cor:100}
Any $\a$-adically complete $A$-module is a quotient of an $\a$-adically free 
$A$-module.
\end{cor}

\begin{cor}[{\cite[Proposition 3.13]{Ye2}}]  \label{cor:51}
An $A$-module is $\a$-adically projective if and only if
it is a direct summand of some $\a$-adically free module.
\end{cor}

\begin{cor} \label{cor:50}
\begin{enumerate}
\item Any  $\a$-adically projective $A$-module is flat.

\item Any $\a$-adically complete $A$-module is a quotient of an
$\a$-adically projective $A$-module.

\item If $P$ is a projective $A$-module then its $\a$-adic completion
$\what{P}$ is $\a$-adically projective.
\end{enumerate}
\end{cor}

\begin{proof}
Combine Theorem \ref{thm:70} and Corollaries \ref{cor:100} and \ref{cor:51}
\end{proof}

For $k \in \N$ let us write $A_k := A / \a^{k+1}$.

\begin{cor} \label{cor:105}
Let $M$ be any $A$-module, and let $\what{M}$ be its $\a$-adic completion. 
\begin{enumerate}
\item The $A$-module $\what{M}$ is $\a$-adically complete.

\item For any $k \in \N$ the homomorphism 
$\tau_{M, k} :  A_k \ot_A M \to A_k \ot_A \what{M}$,
that is induced by $\tau_{M} :  M \to \what{M}$, is bijective.
\end{enumerate}
\end{cor}

\begin{proof}
(1) is \cite[Corollary 3.5]{Ye2}. As for (2): in the course of the proof of 
\cite[Theorem 1.2]{Ye2} it is shown that the homomorphism 
$\tau_{M, k}$ is an injection (this is true 
regardless if the ideal $\a$ is finitely generated); and the implication 
(i) $\Rightarrow$ (ii) of loc.\ cit., together with item (1) above, shows that 
$\tau_{M, k}$ is surjective. 
\end{proof}

In \cite[Section 0$_{\mrm{IV}}$.19.2]{EGA-IV}  the notion of {\em formally 
projective} module is defined. Here is what it means in our situation, 
where every $A$-module is given its $\a$-adic topology. An $A$-module $P$ is 
formally projective if for every 
$k$, the $A_k$-module $P_k := A_k \ot_A P$ is projective. 
The next results compare the two notions of projective modules. 

\begin{thm} \label{thm:108}
The following conditions are equivalent for an $\a$-adically complete 
$A$-module $P$~\tup{:}
\begin{enumerate}
\rmitem{i} $P$ is  $\a$-adically projective.
\rmitem{ii} $P$ is formally projective.
\end{enumerate}
\end{thm}

\begin{proof}
(i) $\Rightarrow$ (ii): This is trivial, since any $A_k$-module is an 
$\a$-adically complete $A$-module. 

\medskip \noindent
(ii) $\Rightarrow$ (i):
According to Corollary \ref{cor:100} there is a surjection 
$\alpha : \opn{F}_{\mrm{dec}}(Z, \what{A}) \to P$ for some set $Z$. We will 
prove that $\alpha$ is split; and then by Corollary \ref{cor:51} the 
$A$-module $P$ is $\a$-adically projective.

To obtain a splitting $\beta$ of $\alpha$ we will construct an inverse system 
of splittings $\beta_k : P_k \to \opn{F}_{\mrm{fin}}(Z, A_k)$ of the 
surjections $\alpha_k : \opn{F}_{\mrm{fin}}(Z, A_k) \to P_k$ induced by 
$\alpha$. For $k = 0$ this is easy, since $P_0$ is a projective $A_0$-module. 

Now suppose we have a compatible system of splittings 
$\beta_0, \ldots, \beta_k$. Let us choose any splitting 
$\beta'_{k+1} : P_{k+1} \to \opn{F}_{\mrm{fin}}(Z, A_{k+1})$.
Write 
$L_{k+1} := \opn{Ker}(\alpha_{k+1})$ and 
$L_{k} := \opn{Ker}(\alpha_k)$.
Consider the commutative diagram (solid arrows) 
\[ \UseTips \xymatrix @C=7ex @R=6ex {
0 
\ar[r]
&
L_{k+1}
\ar[r]
\ar[d]_{ \phi_{k+1, k} }
&
\opn{F}_{\mrm{fin}}(Z, A_{k+1})
\ar[r]^(0.6){\alpha_{k+1}}
\ar[d]_{ \theta^{\opn{F}}_{k+1, k} }
&
P_{k+1}
\ar[r]
\ar[d]_{ \theta^{P}_{k+1, k} }
\ar@(ul,ur)@{-->}[l]_{ \beta'_{k+1} }
&
0
\\
0 
\ar[r]
&
L_{k}
\ar[r]
&
\opn{F}_{\mrm{fin}}(Z, A_{k})
\ar[r]^(0.6){\alpha_{k}}
&
P_{k}
\ar[r]
\ar@(dl,dr)@{-->}[l]^{ \beta_k }
&
0
} \]
with exact rows, where $\theta^{P}_{k+1, k}$ and $\theta^{\opn{F}}_{k+1, k}$ 
are the canonical surjections induced by the ring surjection
$\theta_{k+1, k} : A_{k+1} \to A_k$, and $\phi_{k+1, k}$ is the unique 
homomorphism that agrees with the rest. There is a homomorphism 
$\chi : P_{k+1} \to \opn{F}_{\mrm{fin}}(Z, A_{k})$, 
$\chi := \beta_k \circ \theta^{P}_{k+1, k} - \theta^{\opn{F}}_{k+1, k} \circ 
\beta'_{k+1}$.
But $\alpha_k \circ \chi = 0$, and hence 
$\chi : P_{k+1} \to L_k$. 

Because the homomorphism $\alpha_k$ is gotten from $\alpha_{k+1}$ by applying 
the right exact functor $A_k \ot_{A_{k+1}} -$, it follows that 
$\phi_{k+1, k} : L_{k+1} \to L_k$ is surjective.
Since $P_{k+1}$ is a projective 
$A_{k+1}$-module, there is a homomorphism 
$\til{\chi} : P_{k+1} \to L_{k+1}$ lifting $\chi$, 
namely $\phi_{k+1, k} \circ \til{\chi} = \chi$. 
The homomorphism 
$\beta_{k+1} : P_{k+1} \to \opn{F}_{\mrm{fin}}(Z, A_{k+1})$,
$\beta_{k+1} := \beta'_{k+1} - \til{\chi}$,
is a splitting of $\alpha_{k+1}$ that's compatible with $\beta_k$.
\end{proof}

\begin{cor}
An $A$-module $P$ is formally projective iff its $\a$-adic completion 
$\what{P}$ is $\a$-adically projective.
\end{cor}

\begin{proof}
By Corollary \ref{cor:105} the $A$-module $\what{P}$ is $\a$-adically complete,
and $A_k \ot_A P \cong A_k \ot_A \what{P}$ for every $k$.
Now use Theorem \ref{thm:108}.
\end{proof}

Let $M = \bigoplus_{i \in \Z} M^i$ be a graded $A$-module. We define
\begin{equation}\label{eqn:101}
\inf(M) := \inf\{i \left| \right. M^i\ne 0\} \in \Z \cup \{\pm \infty \}
\end{equation}
and
\begin{equation}\label{eqn:102}
\sup(M) := \sup\{i \left| \right. M^i\ne 0\} \in \Z \cup \{\pm \infty \} .
\end{equation}
The amplitude of $M$ is
\begin{equation}\label{eqn:103}
\opn{amp}(M) := \sup(M) - \inf(M) \in \N \cup \{\pm \infty \}.
\end{equation}
(For $M=0$ this reads $\inf(M) = \infty$, $\sup(M) = -\infty$ and
$\opn{amp}(M) = -\infty$.) Thus $M$ is bounded iff $\opn{amp}(M) < \infty$.
For
$M \in \cat{D}(\cat{Mod} A)$ we write
$\mrm{H}(M) :=  \bigoplus_{i \in \Z} \mrm{H}^i(M)$.

\begin{thm} \label{thm:280}
The following conditions are equivalent for $M \in \cat{D}(\cat{Mod} A)^-$.
\begin{enumerate}
\rmitem{i} $M$ is $\a$-adically cohomologically complete.

\rmitem{ii} There is an isomorphism $P \cong M$ in
$\cat{D}(\cat{Mod} A)$, where $P$ is a complex of $\a$-adically
free modules, and $\opn{sup}(P) = \opn{sup}(\mrm{H}(M))$.

\rmitem{iii} There is an isomorphism $P \cong M$ in
$\cat{D}(\cat{Mod} A)$, where $P$ is a bounded above complex of $\a$-adically
projective modules.
\end{enumerate}
\end{thm}

Before the proof, let us recall some notation from \cite{PSY1}. 
For $M \in \cat{D}(\cat{Mod} A)$, there is a canonical homomorphism 
$\tau_M : M \to \Lambda_{\a}(M)$ in $\cat{C}(\cat{Mod} A)$, and there are 
canonical morphisms 
$\tau^{\mrm{L}}_M : M \to \mrm{L}\Lambda_{\a}(M)$ and 
$\xi_M : \mrm{L} \Lambda_{\a} (M) \to \Lambda_{\a}(M)$ in 
$\cat{D}(\cat{Mod} A)$. These are related by 
$\tau_M = \xi_M \circ \tau^{\mrm{L}}_M$. By definition, $M$ is cohomologically
complete if $\tau^{\mrm{L}}_M$ is an isomorphism.

\begin{proof}
(i) $\Rightarrow$ (ii): We assume that $M$ is $\a$-adically
cohomologically complete and nonzero.
Choose a free resolution $Q \to M$
in $\cat{C}(\cat{Mod} A)$, i.e.\  a quasi-isomorphism where $Q$ is a bounded 
above  complex of 
free modules, such that
$\opn{sup}(Q) = \lb \opn{sup}(\mrm{H}(M))$. This is standard.
Let $P := \Lambda_{\a} (Q)$, which is  a complex of $\a$-adically
free modules, and $\opn{sup}(P) = \opn{sup}(Q)$.
Because $Q \cong M$ in
$\cat{D}(\cat{Mod} A)$, $Q$ is also $\a$-adically cohomologically
complete, so
$\tau^{\mrm{L}}_Q : Q \to \mrm{L} \Lambda_{\a} (Q)$
is an isomorphism in $\cat{D}(\cat{Mod} A)$. But $Q$ is K-projective, so
$\mrm{L} \Lambda_{\a} (Q) \cong \Lambda_{\a} (Q) = P$.
(This in fact proves that $\tau_Q : Q \to P$ is a quasi-isomorphism!)
We conclude that $M \cong P$ in $\cat{D}(\cat{Mod} A)$.

\medskip \noindent
(ii) $\Rightarrow$ (iii):
This is trivial.

\medskip \noindent
(iii) $\Rightarrow$ (i):
Let $P$ be a bounded above complex of $\a$-adically
projective modules. The idempotence of completion
(see \cite[Corollary 3.6]{Ye2}) implies that
$\tau_P : P \to \Lambda_{\a} (P)$ is an isomorphism in
$\cat{C}(\cat{Mod} A)$.
According to Corollary \ref{cor:50}(1) the complex $P$
is K-flat; therefore $\xi_P : \mrm{L} \Lambda_{\a} (P) \to \Lambda_{\a} (P)$
is an isomorphism  in $\cat{D}(\cat{Mod} A)$.
It follows that $\tau^{\mrm{L}}_P = (\xi_P)^{-1} \circ \tau_P$
is an isomorphism in $\cat{D}(\cat{Mod} A)$. So $P$ is cohomologically complete, 
and hence, so is $M$.
\end{proof}

For any $M$ we denote by $\bsym{1}_M$ the identity automorphism of $M$.

\begin{lem} \label{lem:40}
Let $N$ be an $\a$-adically complete $A$-module, and let $M$ be any
$A$-module. Then the homomorphism
\[  \opn{Hom}(\tau_M, \bsym{1}_N) : \opn{Hom}_A( \Lambda_{\a} (M) , N) \to
\opn{Hom}_A(M , N) \]
induced by $\tau_M$ is bijective.
\end{lem}

\begin{proof}
Given $\phi : M \to N$ consider the homomorphism
\[  \tau_N^{-1} \circ \Lambda_{\a}(\phi) :  \Lambda_{\a} (M) \to  N . \]
This operation is inverse to $\opn{Hom}(\tau_M, \bsym{1}_N)$.
Hence $\opn{Hom}(\tau_M, \bsym{1}_N)$ is bijective.
\end{proof}

\begin{lem} \label{lem:33}
\begin{enumerate}
\item Let
$0 \to P' \to P \to P'' \to 0$
be an exact sequence, with $P$ and $P''$ $\a$-adically projective
modules. Then this sequence is split, and $P'$ is also $\a$-adically projective.

\item Let $P$ be an acyclic bounded above complex of $\a$-adically projective
modules. Then $P$ is null-homotopic.

\item Let $P$ and $Q$ be bounded above complexes of $\a$-adically projective
modules,
and let $\phi : P \to Q$ be a quasi-isomorphism. Then $\phi$ is a homotopy
equivalence.
\end{enumerate}
\end{lem}

\begin{proof}
(1) Since both $P$ and $P''$ are complete, the sequence is split by property
(ii) of Definition \ref{dfn:4}. And it is easy to see that a direct summand of
an $\a$-adically projective module is also $\a$-adically projective.

\medskip \noindent
(2) This is like the usual proof for a complex of projectives, but using part
(1) above. Cf.\ \cite[Lemma 10.4.6]{We}.

\medskip \noindent
(3) Let $L := \opn{cone}(\phi)$, the mapping cone.
This is an acyclic bounded above complex of $\a$-adically projective modules.
By part (2) the complex $L$ is null-homotopic; and hence $\phi$ is a homotopy
equivalence.
\end{proof}

\begin{lem} \label{lem:36}
Let $P$ be a bounded above complex of $\a$-adically projective modules,
and let $M$ be a complex of $\a$-adically complete modules. Then the canonical
morphism
\[ \xi_{P, M} : \opn{Hom}_A(P, M) \to \opn{RHom}_A(P, M) \]
in $\cat{D}(\cat{Mod} A)$ is an isomorphism.
\end{lem}

\begin{proof}
Choose a resolution $\phi : Q \to P$ where $Q$ is a bounded above complex of
projective modules. Since both $P$ and $Q$ are K-flat complexes, it follows that
$\Lambda_{\a}(\phi) : \Lambda_{\a} (Q) \to \Lambda_{\a} (P)$
is also a quasi-isomorphism. But
$\tau_P : P \to \Lambda_{\a} (P)$ is bijective. We get a quasi-isomorphism
\[ \psi := \tau_P^{-1} \circ \Lambda_{\a}(\phi) :
\Lambda_{\a} (Q) \to P , \]
satisfying
$\psi \circ \tau_Q = \phi : Q \to P$.
According to Lemma \ref{lem:33}(3), $\psi$ is a homotopy equivalence.
Hence it induces a quasi-isomorphism
\[ \opn{Hom}(\psi, \bsym{1}_M) :
\opn{Hom}_A(P, M) \to \opn{Hom}_A(\Lambda_{\a} (Q), M) . \]
On the other hand, since $M$ consists of complete modules, by Lemma
\ref{lem:40} we see that the homomorphism
\[ \opn{Hom}(\tau_Q, \bsym{1}_M) :
\opn{Hom}_A(\Lambda_{\a} (Q), M) \to \opn{Hom}_A(Q, M)  \]
is bijective. We conclude that
\[ \opn{Hom}(\phi, \bsym{1}_M) : \opn{Hom}_A(P, M) \to \opn{Hom}_A(Q, M) \]
is a quasi-isomorphism.
But the homomorphism $\opn{Hom}(\phi, \bsym{1}_M)$ represents $\xi_{P, M}$.
\end{proof}

Let us denote by $\cat{AdPr} (A, \a)$ the full subcategory of
$\cat{Mod} A$ consisting of $\a$-adically projective modules. This is an
additive category. There is a corresponding triangulated category
$\cat{K}^-(\cat{AdPr} (A, \a))$, which is a full subcategory of
$\cat{K}(\cat{Mod} A)$.

\begin{thm} \label{thm:56}
The localization functor
$\cat{K}(\cat{Mod} A) \to  \cat{D}(\cat{Mod} A)$ induces an equivalence of
triangulated categories
\[ \cat{K}^-(\cat{AdPr} (A, \a)) \to
\cat{D}(\cat{Mod} A)^-_{\a \tup{-com}} . \]
\end{thm}

\begin{proof}
By Theorem \ref{thm:280}, the category $\cat{D}(\cat{Mod} A)^-_{\a \tup{-com}}$
is the essential image of
$\cat{K}^-(\cat{AdPr} (A, \a))$.
And by Lemma \ref{lem:36} we see that
\[ \mrm{H}^0(\xi_{P, Q}) :
\opn{Hom}_{\cat{K}(\cat{Mod} A)}(P, Q) \to
\opn{Hom}_{\cat{D}(\cat{Mod} A)}(P, Q) \]
is bijective  for any $P, Q \in \cat{K}^-(\cat{AdPr} (A, \a))$.
\end{proof}

\begin{lem}
Let $M$ be an $\a$-adically complete $A$-module. Then there is a
quasi-isomorphism $P \to M$, where $P$ is a complex of
$\a$-adically free $A$-modules with $\sup(P)\le 0$.
\end{lem}

\begin{proof}
First consider any $\a$-adically complete module $N$. The module $N$ is a
complete metric space with respect to the $\a$-adic metric (see
\cite[Section 1]{Ye2}). Suppose $N'$ is a closed $A$-submodule of $N$ (not
necessarily $\a$-adically complete). Choose a collection $\{ n_z \}_{z \in Z}$
of elements of $N'$, indexed by a set $Z$, that generates $N'$ as an $A$-module.
Consider the module $\mrm{F}_{\tup{dec}}(Z, \what{A})$ of decaying
functions with values in $\what{A}$ (see \cite[Section 2]{Ye2}). According to
\cite[Corollary 2.6]{Ye2} there is a homomorphism
$\phi : \mrm{F}_{\tup{dec}}(Z, \what{A}) \to N$ that sends a decaying function
$g : Z \to \what{A}$ to the convergent series
$\sum_{z \in Z} g(z) n_z \in N$.
Because $N'$ is closed it follows that $\phi(g) \in N'$.
Writing $P := \mrm{F}_{\tup{dec}}(Z, \what{A})$,
we have constructed a surjection $\phi : P \to  N'$.
And of course $P$ is an $\a$-adically free module.

We now construct an $\a$-adically free resolution of the $\a$-adically
complete module $M$. By the previous paragraph there is an
$\a$-adically free module $P^0$ and a surjection
$\eta : P^0 \to  M$. The module $N^0 := \opn{Ker}(\eta)$
is a closed submodule of the $\a$-adically
complete module $P^0$. Hence there is an $\a$-adically free module
$P^1$ and a surjection $P^1 \to N^0$. And so on.
\end{proof}

\begin{thm} \label{thm:251}
Let $M \in \cat{D}(\cat{Mod} A)$ be a complex whose cohomology
$\mrm{H}(M) = \bigoplus_{i \in \Z} \mrm{H}^i(M)$ is bounded, and all the
$A$-modules
$\mrm{H}^i(M)$ are $\a$-adically complete. Then $M$ is cohomologically
$\a$-adically complete.
\end{thm}

\begin{proof}
If $\opn{amp}(\mrm{H} (M)) = 0$, then we can assume $M$ is a single
$\a$-adically complete module. By the lemma above and
Theorem \ref{thm:56} we see that
$M \in \cat{D}(\cat{Mod} A)_{\a \tup{-com}}$.

In general the proof is by descending induction on the amplitude of 
$\mrm{H} (M)$,  as in \cite[pages 69-71]{RD}.

Using smart truncation 
we get a distinguished triangle
$M' \to M \to M'' \to M'[1]$
in $\cat{D}(\cat{Mod} A)$, such that
$\mrm{H} (M')$ and $\mrm{H} (M'')$ have smaller amplitudes, and
$\mrm{H} (M') \oplus \mrm{H} (M'') \cong \mrm{H} (M)$.
Thus $\mrm{H}^i (M')$ and $\mrm{H}^i (M'')$ are complete modules.
By the induction hypotheses, $M'$ and $M''$ are in
$\cat{D}(\cat{Mod} A)_{\a \tup{-com}}$.
Since $\cat{D}(\cat{Mod} A)_{\a \tup{-com}}$
is a triangulated subcategory of $\cat{D}(\cat{Mod} A)$, it contains $M$ too.
\end{proof}

Here is an example showing that the converse of the
theorem above is false.

\begin{exa} \label{exa:1}
Let $A := \K[[t]]$, the power series ring in the variable $t$ over a field
$\K$, and $\a := (t)$.
As shown in \cite[Example 3.20]{Ye2}, there is a complex
\[ P = \bigl( \cdots \to 0 \to  P^{-1} \xar{\d}  P^{0} \to 0 \to \cdots \bigr)
\]
in which $P^{-1}$ and $P^0$ are $\a$-adically free $A$-modules
(of countable rank in the adic sense, i.e.\
$P^{-1} \cong P^{0} \cong \mrm{F}_{\tup{dec}}(\N, A)$),
$\mrm{H}^{-1} (P) = 0$, and the module $\mrm{H}^0 (P)$ is {\em not}
$\a$-adically complete. Yet by Theorem \ref{thm:56} the complex $P$
is cohomologically $\a$-adically complete.
\end{exa}

We end this section with a result on the structure of the category of derived
torsion complexes, that is analogous to Theorem \ref{thm:251}.  (The referee 
alerted us that this result was already observed in \cite{Li}, in the proof of 
Proposition 3.5.4(i).) Let us denote by
$\cat{Inj}_{\a \tup{-tor}}$ the full
subcategory of  $\cat{Mod} A$ consisting of $\a$-torsion injective $A$-modules.
This is an additive category.

\begin{lem} \label{lem:37}
Let $I$ be an injective $A$-module. Then $\Gamma_{\a} (I)$ is also an
injective $A$-module.
\end{lem}

\begin{proof}
This is well-known: see \cite[Lemma III.3.2]{Ha1}.
\end{proof}

\begin{prop} \label{prop:14}
The localization functor
$\cat{K}(\cat{Mod} A) \to \cat{D}(\cat{Mod} A)$
 induces an equivalence
\[ \cat{K}^+(\cat{Inj}_{\a \tup{-tor}}) \to
\cat{D}(\cat{Mod} A)^+_{\a \tup{-tor}} . \]
\end{prop}

\begin{proof}
The fact that this is a fully faithful functor is clear, since the complexes in
$\cat{K}^+(\cat{Inj}_{\a \tup{-tor}})$ are K-injective.
We have to prove that this functor is essentially surjective on objects. So
take $M \in \cat{D}(\cat{Mod} A)^+_{\a \tup{-tor}}$, and let
$M \to I$ be a minimal injective resolution of $M$.
We know that all the cohomology modules $\mrm{H}^p(M)$ are $\a$-torsion.
By Lemma \ref{lem:37} it follows that the injective hull of
$\mrm{H}^p(M)$ is also $\a$-torsion. This implies that
$I$ belongs to $\cat{K}^+(\cat{Inj}_{\a \tup{-tor}})$.
\end{proof}

\section{Cohomologically Complete Nakayama} \label{sec:nak}

In this section we prove a cohomologically complete version of the Nakayama
Lemma. This is influenced by the paper \cite{KS}, as explained in  Remark 
\ref{rem:ks} below. Throughout this section we
assume this:

\begin{setup} \label{set:261}
$A$ is a noetherian ring, $\a$-adically complete with respect to
some ideal $\a$. We write $A_0 := A / \a$.
\end{setup}

\begin{thm}[Cohomologically Complete Nakayama] \label{thm:23}
With Setup \tup{\ref{set:261}}, let
$M \in \cat{D}(\cat{Mod} A)_{\a \tup{-com}}$ and $i_0 \in \Z$
be such that $\mrm{H}^i (M) = 0$ for all $i > i_0$, and
$\mrm{H}^{i_0} (A_0 \otimes^{\mrm{L}}_A M)$ is a finitely generated
$A_0$-module. Then $\mrm{H}^{i_0} (M)$ is a finitely generated
$A$-module.
\end{thm}

\begin{proof}
We may assume that $i_0 = 0$.
According to Theorem \ref{thm:280}
we can replace $M$ with a complex $P$ of $\a$-adically free
$A$-modules such that
$\opn{sup}(P) = 0$. There is an exact sequence of $A$-modules
\[  P^{-1} \xar{\d} P^0 \xar{\eta} \mrm{H}^0 (P) \to 0 . \]
Now $A_0 \otimes^{\mrm{L}}_A M \cong A_0 \otimes_A P$
in $\cat{D}(\cat{Mod} A_0)$. Let
$L_0 := \mrm{H}^{0} (A_0 \otimes_A P)$, so we have an exact sequence of
$A_0$-modules
\[  A_0 \otimes_A P^{-1} \xar{\bsym{1}_{A_0} \otimes \, \d} A_0 \otimes_A P^0
\xar{\nu} L_0 \to 0 . \]
Choose a finite collection
$\{ \bar{p}_z \}_{z \in Z}$ of elements of
$A_0 \otimes_A P^0$, such that the collection \linebreak
$\{ \nu(\bar{p}_z) \}_{z \in Z}$ generates $L_0$.
Let
\[ \theta_0 : \opn{F}_{\mrm{fin}}(Z, A_0) \to A_0 \otimes_A P^0 \]
be the homomorphism
corresponding to the collection $\{ \bar{p}_z \}_{z \in Z}$.
Then the homomorphism
\[ \psi_0 := (\bsym{1}_{A_0} \otimes \, \d, \, \theta_0) :
(A_0 \otimes_A P^{-1}) \oplus
\opn{F}_{\mrm{fin}}(Z, A_0) \to A_0 \otimes_A P^0 \]
is surjective.

For any $z \in Z$ choose some element $p_z \in P^0$ lifting the
element $ \bar{p}_z$, and let
$\theta : \opn{F}_{\mrm{fin}}(Z, A) \to P^0$
be the corresponding homomorphism.
We get a homomorphism of $A$-modules
\[ \psi := (\d, \theta) : P^{-1} \oplus
\opn{F}_{\mrm{fin}}(Z, A) \to P^0 . \]
It fits into a commutative diagram
\[ \UseTips \xymatrix @C=7ex @R=6ex {
P^{-1}  \oplus \opn{F}_{\mrm{fin}}(Z, A)
\ar[r]^(0.6){\psi}
\ar[d]_{\rho}
&
P^0
\ar[d]^{\pi}
\\
(A_0 \otimes_A P^{-1})  \oplus \opn{F}_{\mrm{fin}}(Z, A_0)
\ar[r]^(0.65){\psi_0}
&
A_0 \otimes_A P^{0} \ ,
} \]
where $\rho$ and $\pi$ are the canonical surjections induced by $A \to A_0$.
Now
$\psi_0 \circ \rho = \pi \circ \psi$ is surjective.
By the complete Nakayama \cite[Theorem 2.11]{Ye2} the homomorphism $\psi$
is surjective. We conclude that $\mrm{H}^0 (P)$ is generated by the finite
collection $\{ \eta(p_z) \}_{z \in Z}$.
\end{proof}

\begin{rem}
With some extra work (cf.\ proof of Lemma \ref{lem:cof.1}) one can prove the
following stronger result: Let $M \in \cat{D}(\cat{Mod} A)^-_{\a \tup{-com}}$
and $i_0 \in \Z$. Then $\mrm{H}^i (M)$ is finitely generated over $A$ for all
$i \geq i_0$ iff $\mrm{H}^i (A_0 \otimes^{\mrm{L}}_A M)$ is finitely generated
over $A_0$ for all $i \geq i_0$.
\end{rem}

The next result will be used several times. (The easy proof is an exercise.)

\begin{lem}[K\"unneth Trick] \label{lem:nak.1}
Let $M, N \in \cat{D}(\cat{Mod} A)$. If $i\ge \sup(\mrm{H}(M))$ and 
$j \ge \sup(\mrm{H}(N))$, then there is a canonical isomorphism of $A$-modules
\[ \mrm{H}^{i + j} (M \otimes^{\mrm{L}}_A N) \cong
\mrm{H}^{i} (M) \otimes_A \mrm{H}^{j} (N) . \]
\end{lem}

\begin{cor}
Let $M \in \cat{D}(\cat{Mod} A)^-_{\a \tup{-com}}$.
If $A_0 \otimes^{\mrm{L}}_A M = 0$ then $M = 0$.
\end{cor}

\begin{proof}
Let's assume, for the sake of
contradiction,  that $M \neq 0$ but $A_0 \otimes^{\mrm{L}}_A M = 0$.
Let
$i := \sup (\mrm{H} (M))$,
which is an integer, since $M$ is nonzero and bounded above.
By Lemma \ref{lem:nak.1} we know that
\[ \mrm{H}^{i} (A_0 \otimes^{\mrm{L}}_A M) \cong
A_0 \otimes_A \mrm{H}^{i} (M) ; \]
therefore
$A_0 \otimes_A \mrm{H}^{i} (M) = 0$.
Now Theorem \ref{thm:23}
says that the $A$-module $\mrm{H}^{i} (M)$ is finitely generated.
So by the usual Nakayama Lemma we conclude that $\mrm{H}^{i} (M) = 0$. This is
a contradiction.
\end{proof}

\begin{rem}\label{rem:ks}
A triangulated functor $F$ with the property that $F(M)=0$ if and only if $M=0$ 
is called a {\em conservative functor} in \cite[Section 1.4]{KS}.
The corollary says that the functor
\[ A_0 \otimes^{\mrm{L}}_A - : \cat{D}(\cat{Mod} A)^- \to
\cat{D}(\cat{Mod} A_0)^- \]
is conservative.

Let $\bsym{a} = (a_1, \ldots, a_n)$ be a generating sequence for the ideal $\a$,
and let $K := \opn{K}(A; \bsym{a})$, the Koszul complex, which we view as a DG
$A$-algebra. By arguments similar to those used in \cite{PSY2}, one
can show that the functor
\[ K \otimes^{\mrm{L}}_A - : \cat{D}(\cat{Mod} A) \to
\til{\cat{D}}(\cat{DGMod} K) \]
is conservative. If $\bsym{a}$ is a regular sequence then the DG algebra
homomorphism $K \to A_0$ is a quasi-isomorphism; and hence the functor
$A_0 \otimes^{\mrm{L}}_A -$ is conservative on unbounded complexes. This was
proved in \cite[Corollary 1.5.9]{KS} in the principal case ($n = 1$).
\end{rem}

\section{Cohomologically Cofinite Complexes} \label{sec:cof}

We continue with Setup \ref{set:261}. Recall that
$\cat{D}(\cat{Mod} A)^{\mrm{b}}_{\a \tup{-com}}$
is the category of bounded cohomologically $\a$-adically complete complexes.

\begin{prop} \label{prop:263}
The category $\cat{D}_{\mrm{f}}(\cat{Mod} A)^{\mrm{b}}$ is contained in
$\cat{D}(\cat{Mod} A)^{\mrm{b}}_{\a \tup{-com}}$.
\end{prop}

\begin{proof}
Any finitely generated $A$-module is $\a$-adically complete. So this is a
special case of Theorem \ref{thm:251}.
\end{proof}

\begin{dfn} \label{dfn:3}
A complex $M \in \cat{D}(\cat{Mod} A)^{\mrm{b}}$
is called {\em cohomologically $\a$-adically cofinite} if
$M \cong \mrm{R} \Gamma_{\a} (N)$
for some
$N \in \cat{D}_{\mrm{f}}(\cat{Mod} A)^{\mrm{b}}$.

We denote by
$\cat{D}(\cat{Mod} A)^{\mrm{b}}_{\a \tup{-cof}}$
the full subcategory of $\cat{D}(\cat{Mod} A)^{\mrm{b}}$ consisting of
cohomologically $\a$-adically cofinite complexes.
\end{dfn}

See Example \ref{exa:260} for an explanation of the name ``cofinite''.

Here is one characterization of cohomologically $\a$-adically cofinite
complexes. It was previously proved in \cite[Section 2.5]{AJL2}.

\begin{prop} \label{prop:3}
The following conditions are equivalent for
$M \in \cat{D}(\cat{Mod} A)^{\mrm{b}}_{\a \tup{-tor}}$~\tup{:}
\begin{enumerate}
\rmitem{i} $M$ is in $\cat{D}(\cat{Mod} A)^{\mrm{b}}_{\a \tup{-cof}}$.
\rmitem{ii} The complex $\mrm{L} \Lambda_{\a} (M)$ is in
$\cat{D}_{\mrm{f}}(\cat{Mod} A)^{\mrm{b}}$.
\end{enumerate}
\end{prop}

\begin{proof}
Let $N := \mrm{L} \Lambda_{\a} (M)$. By MGM Equivalence
\cite[Theorem 0.3]{PSY1}  we know that
$N \in \cat{D}(\cat{Mod} A)^{\mrm{b}}_{\a \tup{-com}}$, and that
$M \cong \mrm{R} \Gamma_{\a} (N)$.
Moreover, if
$M \cong \mrm{R} \Gamma_{\a} (N')$ for some other
$N' \in \cat{D}(\cat{Mod} A)^{\mrm{b}}_{\a \tup{-com}}$, then
$N' \cong N$. Thus
$M \in \cat{D}(\cat{Mod} A)^{\mrm{b}}_{\a \tup{-cof}}$
if and only if
$N \in \cat{D}_{\mrm{f}}(\cat{Mod} A)^{\mrm{b}}$.
\end{proof}

\begin{cor} \label{cor:cof.1}
The functor $\mrm{R} \Gamma_{\a}$ induces an equivalence of triangulated
categories
\[ \cat{D}_{\mrm{f}}(\cat{Mod} A)^{\mrm{b}} \to
\cat{D}(\cat{Mod} A)^{\mrm{b}}_{\a \tup{-cof}} \ , \]
with quasi-inverse $\mrm{L} \Lambda_{\a}$.
\end{cor}

\begin{proof}
Immediate from MGM Equivalence \cite[Theorem 0.3]{PSY1} and Proposition
\ref{prop:3}.
\end{proof}

\begin{rem} \label{rem:100}
Let  $\cat{D}(\cat{Mod} A)_{\a \tup{-cof}}$ be the essential image under 
$\mrm{R} \Gamma_{\a}$ of $\cat{D}_{\tup{f}}(\cat{Mod} A)$.
The category $\cat{D}(\cat{Mod} A)_{\a \tup{-cof}}$
was first considered  in \cite[Section 2.5]{AJL2}, where the notation 
$\cat{D}_{\mrm{c}}^*$ was used.

The category $\cat{D}(\cat{Mod} A)^{\mrm{b}}_{\a \tup{-cof}}$
is important because it contains the {\em t-dualizing complexes}. 
Recall that a t-dualizing complex is a complex 
$R \in \cat{D}(\cat{Mod} A)^{\mrm{b}}_{\a \tup{-tor}}$ that has finite 
injective dimension, $\opn{Ext}_A^i(A_0,R)$ is a finitely generated $A$-module 
for all $i$, and the adjunction morphism $A \to \opn{RHom}_A(R,R)$ is 
an isomorphism. See \cite[Definition 5.2]{Ye1}, and \cite[Definition 
2.5.1]{AJL2}. In \cite[Definition 5.2]{Ye1} we used the name ``dualizing 
complex'' for ``t-dualizing complex'' in the adic case; but that usage is now 
obsolete. By \cite[Proposition 2.5.8]{AJL2}, if $R$ is a t-dualizing complex, 
then the adjunction morphism 
\[ M\to \opn{RHom}_A( \opn{RHom}_A(M,R),R) \] 
is an isomorphism for any 
$M \in \cat{D}(\cat{Mod} A)_{\a \tup{-cof}}$. 

The characterization of cohomologically $\a$-adically cofinite
complexes in Proposition \ref{prop:3} is not very practical, since it is very
hard to compute $\mrm{L} \Lambda_{\a} (M)$. Another characterization of the
category $\cat{D}(\cat{Mod} A)^{\mrm{b}}_{\a \tup{-cof}}$
was proposed in \cite[Problem 5.7]{Ye1}; but at the time we could not prove
that it is correct. This is solved in
Theorem \ref{thm:cof.1} below.
\end{rem}

\begin{lem} \label{lem:31}
Let $L, K \in \cat{D}(\cat{Mod} A)^{\mrm{b}}$. Assume that
$\opn{Ext}^i_A(A_0, L)$ and $\mrm{H}^i (K)$ are finitely generated $A_0$-modules
for all $i$. Then $\opn{Ext}^i_A(K, L)$
are finitely generated $A$-modules for all $i$.
\end{lem}

\begin{proof}
Step 1. Suppose $K$ is a single $A$-module (sitting in degree $0$).
Then $K$ is a finitely generated $A_0$-module. Define
\[ M := \opn{RHom}_A(A_0, L) \in \cat{D}(\cat{Mod} A_0)^{+} . \]
By Hom-tensor adjunction we get
\[ \opn{RHom}_A(K, L) \cong
\opn{RHom}_{A_0}(K, \opn{RHom}_A(A_0, L)) =
\opn{RHom}_{A_0}(K, M )  \]
in $\cat{D}(\cat{Mod} A_0)^+$.
But the assumption is that
$M \in \cat{D}_{\tup{f}}(\cat{Mod} A_0)^{+}$;
and hence we also have
\[ \opn{RHom}_{A_0}(K, M) \in  \cat{D}_{\mrm{f}}(\cat{Mod} A_0)^{+} . \]
This shows that $\opn{Ext}^i_A(K, L)$ are finitely generated $A_0$-modules.

\medskip \noindent
Step 2. Now $K$ is a bounded complex, and $\mrm{H}^i (K)$ are finitely generated
$A_0$-modules for all $i$. The proof is by induction on the amplitude of
$\mrm{H} (K)$. The induction starts with $\opn{amp} (\mrm{H} (K)) = 0$, and this
is covered by Step 1. If
$\opn{amp} (\mrm{H} (K)) > 0$, then using smart truncation
(as in the proof of Theorem \ref{thm:251}) we construct a
distinguished triangle
$K' \to K \to K'' \to K'[1]$ in $\cat{D}(\cat{Mod} A)$,
where $\mrm{H} (K')$ and $\mrm{H} (K'')$ have smaller amplitudes, and
$\mrm{H}^i (K')$ and $\mrm{H}^i (K'')$ are finitely generated $A_0$-modules for
all $j$. By applying $\opn{RHom}_A(-, L)$ to the triangle above we obtain a
distinguished triangle
\[ \opn{RHom}_A(K'', L) \to \opn{RHom}_A(K, L) \to
\opn{RHom}_A(K', L) \to , \]
and hence a long exact sequence
\[ \cdots \to \opn{Ext}^i_A(K'', L) \to \opn{Ext}^i_A(K, L)
 \to \opn{Ext}^i_A(K', L) \to  \cdots . \]
of $A$-modules. {}From this we conclude that $\opn{Ext}^i_A(K, L)$ are finitely
generated (and $\a$-torsion) $A$-modules.
\end{proof}

\begin{lem} \label{lem:4}
Let $L \in \cat{D}(\cat{Mod} A)^{\mrm{b}}$ and $i_0 \in \Z$. Assume that
$\mrm{H}^i (L) = 0$ for all $i > i_0$, and that
$\opn{Ext}^i_A(A_0, L)$ is finitely generated over $A_0$ for all $i$.
Then $\mrm{H}^{i_0}  (A_0 \otimes^{\mrm{L}}_A L)$
is finitely generated over $A_0$.
\end{lem}

\begin{proof}
It is clear that $\mrm{H}^{i_0}  (A_0 \otimes^{\mrm{L}}_A L)$
is an $A_0$-module. We have to prove that it is finitely generated as
$A$-module.

Choose a generating sequence $\bsym{a} = (a_1, \ldots, a_n)$ of the ideal
$\a$. Let $K := \opn{K}(A, \bsym{a})$ be the Koszul complex.
We know that $K$ is a bounded complex of finitely generated free $A$-modules;
the cohomologies $\mrm{H}^i (K)$ are all finitely generated
$A_0$-modules; they vanish unless $-n \leq i \leq 0$; and
$\mrm{H}^0 (K) \cong A_0$. Also $K$ has the self-duality property
$K^{\vee} \cong K[-n]$, where
$K^{\vee} := \opn{Hom}_A (K, A)$.

Let us consider the complex
$M := \opn{Hom}_A (K, L)$.
By Lemma \ref{lem:31} we know that $\mrm{H}^i (M)$ are all finitely generated
$A$-modules. But there is also an isomorphism of complexes
$M \cong K^{\vee} \otimes_A L$.
By the K\"unneth trick (Lemma \ref{lem:nak.1}) we conclude that
\[ \begin{aligned}
& \mrm{H}^{n + i_0} (M) \cong
\mrm{H}^{n} (K^{\vee}) \otimes_A \mrm{H}^{i_0} (L)
\cong \mrm{H}^{0} (K) \otimes_A \mrm{H}^{i_0} (L) \\
& \qquad \cong A_0 \otimes_A \mrm{H}^{i_0} (L)
 \cong  \mrm{H}^{i_0} (A_0 \otimes^{\mrm{L}}_A  L ) .
\end{aligned} \]
So
$\mrm{H}^{i_0} (A_0 \otimes^{\mrm{L}}_A  L )$
is a finitely generated $A$-module.
\end{proof}

\begin{lem} \label{lem:cof.1}
Let $N \in \cat{D}(\cat{Mod} A)^{\mrm{b}}_{\a \tup{-com}}$. The following two
conditions are equivalent:
\begin{enumerate}
\rmitem{i} For every $j \in \Z$ the $A$-module $\mrm{H}^j (N)$ is
finitely generated.
\rmitem{ii} For every $j \in \Z$ the $A_0$-module
$\opn{Ext}^j_A(A_0, N)$ is finitely generated.
\end{enumerate}
\end{lem}

\begin{proof}
(i) $\Rightarrow$ (ii): It suffices to prove that
$\opn{Ext}^j_A(A_0, N)$ are finitely generated $A$-modules for all $j$.
This follows from \cite[Proposition II.3.3]{RD}.

\medskip \noindent
(ii) $\Rightarrow$ (i):
The converse is more difficult. Let us choose an
integer $i_0$ such that $\mrm{H}^i (N) = 0$ for all $i > i_0$.
We are going to prove that $\mrm{H}^i (N)$ is finitely generated by
descending induction on $i$  (as in the proof of Theorem \ref{thm:251}), 
starting from $i = i_0 + 1$ (which is trivial of
course). So let's suppose that $\mrm{H}^j (N)$ is finitely generated for all
$j > i$, and we shall prove that
$\mrm{H}^i (N)$ is also finitely generated.
Using smart truncation of $N$ at $i$ (cf.\ \cite[pages 69-70]{RD}),  there is a 
distinguished triangle
\begin{equation} \label{eqn:25}
L \xar{\phi} N \xar{\psi} M \to L[1]
\end{equation}
in $\cat{D}(\cat{Mod} A)$, such that: $\mrm{H}^j (L) = 0$ and
$\mrm{H}^j(\psi) : \mrm{H}^j (N) \to \mrm{H}^j (M)$
is bijective for all $j > i$; and
$\mrm{H}^j (M) = 0$ and
$\mrm{H}^j(\phi) : \mrm{H}^j (L) \to \mrm{H}^j (N)$
is bijective for all $j \leq i$.
By the induction hypothesis the bounded complex $M$ has finitely generated
cohomologies; so by Proposition \ref{prop:263} it
is cohomologically complete. Since $N$ is also
cohomologically complete, and $\cat{D}(\cat{Mod} A)^{\mrm{b}}_{\a \tup{-com}}$
is a triangulated category, it follows that $L$ is
cohomologically complete too.

We know from the implication ``(i) $\Rightarrow$ (ii)'', applied to $M$, that
$\opn{Ext}^j_A(A_0, M)$
is a finitely generated $A_0$-module for every $j$. The exact sequence
\[ \opn{Ext}^{j-1}_A(A_0, M) \to
\opn{Ext}^{j}_A(A_0, L) \to
\opn{Ext}^{j}_A(A_0, N) \]
coming from (\ref{eqn:25}) shows that $\opn{Ext}^j_A(A_0, L)$ is also finitely
generated for every $j$. So according to Lemma \ref{lem:4} the $A_0$-module
$\mrm{H}^{i}(A_0 \otimes^{\mrm{L}}_A L)$ is finitely generated. We can now use
Theorem \ref{thm:23} to conclude that the $A$-module
$\mrm{H}^{i} (L)$ is finitely generated. But
$\mrm{H}^{i} (L) \cong \mrm{H}^{i} (N)$.
\end{proof}

The main result of this section is this:

\begin{thm} \label{thm:cof.1}
Assuming Setup \tup{\ref{set:261}}, let
$M \in \cat{D}(\cat{Mod} A)^{\mrm{b}}_{\a \tup{-tor}}$. The following two
conditions are equivalent\tup{:}
\begin{enumerate}
\rmitem{i} $M$ is cohomologically $\a$-adically cofinite.
\rmitem{ii} For every $j \in \Z$ the $A_0$-module
$\opn{Ext}^j_A(A_0, M)$ is finitely generated.
\end{enumerate}
\end{thm}

\begin{proof}
Let $N := \mrm{L} \Lambda_{\a} (M)$, so
$N \in \cat{D}(\cat{Mod} A)^{\mrm{b}}_{\a \tup{-com}}$,
and according to Proposition \ref{prop:3} we know that
$N \in \cat{D}_{\mrm{f}}(\cat{Mod} A)^{\mrm{b}}$ if and only if
$M \in \cat{D}(\cat{Mod} A)^{\mrm{b}}_{\a \tup{-cof}}$.
In other words, condition (i) above is equivalent to condition
(i) of Lemma \ref{lem:cof.1}.

On the other hand, since $A_0 \cong \mrm{L} \Lambda_{\a} (A_0)$,
by MGM Equivalence we have
\[ \opn{Ext}^j_A(A_0, M) \cong \opn{Hom}_{\cat{D}(A)}(A_0, M[j]) \cong
\opn{Hom}_{\cat{D}(A)}(A_0, N[j]) \cong \opn{Ext}^j_A(A_0, N) , \]
where $\cat{D}(A) := \cat{D}(\cat{Mod} A)$. So
condition (ii) above is equivalent to condition
(ii) of Lemma \ref{lem:cof.1}.
\end{proof}

\begin{rem}
In the notations of this paper, Lemma 2.5.3 in \cite{AJL2} states that if $R$ is 
 a t-dualizing complex, and if $R$ is cohomologically $\a$-adically cofinite, 
then $\mrm{L} \Lambda_{\a}(R)$ is a c-dualizing complex. The above theorem shows 
that the assumption that $R$ is cohomologically $\a$-adically cofinite is 
redundant. 
\end{rem}

Suppose the ring $A$ admits a t-dualizing complex $R^{\mrm{t}}$; 
see Remark \ref{rem:100}. Following Hartshorne \cite{Ha2} 
we say that a complex $M \in \cat{D}(\cat{Mod} A)$ is {\em cofinite} if 
$M \cong \opn{RHom}_A(N, R^{\mrm{t}})$
for some $N \in \cat{D}_{\mrm{f}}(\cat{Mod} A)$.
Proposition \ref{prop:ha2} below shows that for bounded complexes
the two notions of cofiniteness (ours and that of \cite{Ha2}) coincide. 

\begin{lem} \label{lem:101}
There is a bifunctorial isomorphism 
\[ \mrm{R} \Ga_{\a} (\opn{RHom}_A(M , N)) \cong 
\opn{RHom}_A(M , \mrm{R} \Ga_{\a} (N)) \]
for $M \in \cat{D}_{\mrm{f}}(\cat{Mod} A)^{-}$ and 
$N \in \cat{D}(\cat{Mod} A)^+$.
\end{lem}

\begin{proof}
Let $N \to I$ be a bounded below injective resolution, and let 
$P \to N$ be a bounded above resolution consisting of finitely generated 
projective $A$-modules. We get 
\[ \begin{aligned}
&  \mrm{R} \Ga_{\a} (\opn{RHom}_A(M , N)) \cong 
\Gamma_{\a}(\opn{Hom}_A(P, I)) \\
& \qquad = \opn{Hom}_A(P, \Gamma_{\a}(I))
\cong \opn{RHom}_A(M , \mrm{R} \Ga_{\a} (N)) . 
\end{aligned} \]
\end{proof}

\begin{prop} \label{prop:ha2}
Assuming Setup \tup{\ref{set:261}}, suppose that $R^{\mrm{t}}$ is a t-dualizing 
complex over $A$. The following two conditions are equivalent for 
$M \in \cat{D}(\cat{Mod} A)^{\mrm{b}}$\tup{:}
\begin{enumerate}
\rmitem {i} The complex $M$ is cohomologically $\a$-adically cofinite, namely 
$M \cong \mrm{R}\Gamma_{\a}(L)$ for some 
$L \in \cat{D}_{\mrm{f}}(\cat{Mod} A)^{\mrm{b}}$.
\rmitem {ii} The complex $M$ is cofinite, namely there is a complex 
$N \in \cat{D}_{\mrm{f}}(\cat{Mod} A)$ 
such that $M \cong \opn{RHom}_A(N, R^{\mrm{t}})$.
\end{enumerate}

Moreover, when these conditions hold, the complex $N$ has bounded cohomology, 
and $N \cong  \opn{RHom}_A(L, R)$, where 
$R := \mrm{L} \Lambda_{\a}(R^{\mrm{t}})$.
\end{prop}

\begin{proof}
We first observe that $R = \mrm{L} \Lambda_{\a}(R^{\mrm{t}})$
is in $\cat{D}_{\mrm{f}}(\cat{Mod} A)^{\mrm{b}}$, and it is a c-dualizing 
complex over $A$ (i.e.\ a dualizing complex in the usual sense).
See \cite[Lemma 2.5.3(a)]{AJL2}.
Let us write 
$D^{\mrm{t}}(-) := \opn{RHom}_A(-, R^{\mrm{t}})$
and 
$D(-) := \opn{RHom}_A(-, R)$.

\medskip \noindent
(i) $\Rightarrow$ (ii): 
We know that $L \cong \mrm{L} \Lambda_{\a}(M)$.
Define $N := D(L) \in  \cat{D}_{\mrm{f}}(\cat{Mod} A)^{\mrm{b}}$.
By \cite[Proposition 2.5.8(a)]{AJL2} the adjunction 
morphism $M \to D^{\mrm{t}}(D^{\mrm{t}}(M))$
is an isomorphism. On the other hand, according to GM duality 
\cite[Theorem 7.12]{PSY1}, there is an isomorphism
\[ D^{\mrm{t}}(M) = \opn{RHom}_A(M, R^{\mrm{t}}) 
\cong \opn{RHom}_A(L, R) = D(L) \cong N . \]
We conclude that 
$M \cong D^{\mrm{t}}(N)$. 

\medskip \noindent
(ii) $\Rightarrow$ (i): 
By \cite[Proposition 2.5.8(c)]{AJL2} the adjunction 
morphism $N \to \lb D^{\mrm{t}}(D^{\mrm{t}}(N))$
is an isomorphism; and this implies that $N$ has bounded cohomology, i.e.\ 
$N \in \cat{D}_{\mrm{f}}(\cat{Mod} A)^{\mrm{b}}$.
Let $L := D(N) \in  \cat{D}_{\mrm{f}}(\cat{Mod} A)^{\mrm{b}}$.
Then, using Lemma \ref{lem:101}, and the isomorphism 
$M \cong D^{\mrm{t}}(D^{\mrm{t}}(M))$, we have
\[ \mrm{R} \Ga_{\a}(L) = \mrm{R} \Ga_{\a} (\opn{RHom}_A(N , R))
\cong \opn{RHom}_A(N , R^{\mrm{t}}) = D^{\mrm{t}}(N) \cong M . \]
\end{proof}

\begin{rem}
In \cite{Ha2}, Hartshorne considered a ring $A$ as in Setup \ref{set:261}, 
which is also regular of finite Krull dimension.
Then the complex $R^{\mrm{t}} := \mrm{R}\Gamma_{\a}(A)$ is a t-dualizing 
complex. In this situation, and in view of Proposition \ref{prop:ha2},
we see that our Theorem \ref{thm:cof.1} coincides with 
\cite[Theorem 5.1]{Ha2} for complexes with bounded cohomology.

It is quite possible that the method of proof in \cite{Ha2} can be extended to 
any ring $A$ that admits a dualizing complex $R$ (and then 
$R^{\mrm{t}} := \mrm{R}\Gamma_{\a}(R)$ is a t-dualizing 
complex). But we did not try to do this. 

Note that our Theorem \ref{thm:cof.1}  does not require $A$ to have a dualizing 
complex, nor even to have finite Krull dimension. 
\end{rem}

For a local ring the category $\cat{D}(\cat{Mod} A)^{\mrm{b}}_{\a \tup{-cof}}$
is actually easy to describe, using Theorem \ref{thm:cof.1}. This was also noted 
by  Hartshorne in \cite{Ha2}.

\begin{exa} \label{exa:260}
Suppose $A$ is local and $\m := \a$ is its maximal ideal. An $A$-module is
called {\em cofinite} if it is artinian. We denote by
$\cat{Mod}_{\tup{cof}} A$
the category of cofinite modules. Let $J(\m)$ be an injective hull of the
residue field $A_0 = A / \m$. Then $J(\m)$ is the only indecomposable injective 
$\m$-torsion $A$-module (up to isomorphism). {\em Matlis duality} \cite{Ma} 
says that
\begin{equation} \label{eqn:26}
\opn{Hom}_A(-, J(\m)) : \cat{Mod}_{\mrm{f}} A \to \cat{Mod}_{\tup{cof}} A
\end{equation}
is a duality (contravariant equivalence).

Let
$M \in \cat{D}( \cat{Mod} A)^{\mrm{b}}_{\m \tup{-tor}}$, and let $M \to I$
be its minimal injective resolution. The bounded below complex of injectives
$I = \bigl( \cdots \to I^0 \to I^1 \to  \cdots \bigr)$
has this structure:
$I^q \cong J(\m)^{\oplus \mu_q}$,
where $\mu_q$ are the {\em Bass numbers}, that in general could
be infinite cardinals. The Bass numbers satisfy the equation
$\mu_q = \opn{rank}_{A_0} (\opn{Ext}^q_A(A_0, M))$.
By Theorem \ref{thm:cof.1} we know that
$M \in \cat{D}(\cat{Mod} A)^{\mrm{b}}_{\m \tup{-cof}}$
if and only if $\mu_q < \infty$ for all $q$. On the other hand, from
(\ref{eqn:26}) we see
that an $\m$-torsion module $M$ has finite Bass numbers if and only if it is
cofinite.  We conclude that cofinite modules are cohomologically cofinite,
and the inclusion
\[ \cat{D}^{\mrm{b}}(\cat{Mod}_{\tup{cof}} A) \to
\cat{D}( \cat{Mod} A)^{\mrm{b}}_{\m \tup{-cof}} \]
is an equivalence.

Note that the module $J(\m)$ is a {\em t-dualizing complex}
over $A$.
\end{exa}


\end{document}